\documentclass[10pt]{amsart}
\usepackage{amsmath,enumerate, xcolor}
\usepackage{amsfonts}
\usepackage{tikz}
\usepackage{amssymb}
\usepackage{amsxtra}
\usepackage{amsthm}
\usepackage{amsmath,amssymb,amsfonts,amsthm,color,latexsym}
\usepackage{xcolor}
\usepackage{float}

\newcommand{\Z}{\mathbb{Z}}

\newcommand{\C}{\mathbb{C}}
\newcommand{\F}{\mathcal{F}}

\newcommand{\B}{\mathcal{B}}

\newcommand{\ord}{\text{ord}}

\newtheorem{theorem}{Theorem}[section]
\newtheorem{example}[theorem]{Example}
\newtheorem{remark}[theorem]{Remark} 

\newtheorem{lemma}[theorem]{Lemma}
\newtheorem{proposition}[theorem]{Proposition}
\newtheorem{corollary}[theorem]{Corollary}
\newtheorem{definition}[theorem]{Definition}
\usepackage[pagebackref]{hyperref}
\newcommand{\mr}[1]{\ensuremath{\mathrm{#1}}}
\newcommand{\mc}[1]{\ensuremath{\mathcal{#1}}}
\newcommand{\mb}[1]{\ensuremath{\mathbb{#1}}}

\title{Indices of holomorphic foliations and the bifurcation conjecture}

\date{\today}

\author[M. Falla Luza]{Maycol Falla Luza}
\address[M. Falla Luza]{Instituto de Matem\'atica e Estat\'istica, Universidade Federal Fluminense, Rua Professor Marcos Waldemar de Freitas Reis, s/n, CEP 24210-201\\
Bloco H, Campus do Gragoat\'a - Niter\'oi - RJ, Brasil}
\email{hfalla@id.uff.br}

\author[A. Fern\'andez-P\'erez]{Arturo Fern\'andez-P\'erez}
\address[A. Fern\'andez-P\'erez]{Department of Mathematics, Federal University of Minas Gerais, Av. Ant\^onio Carlos, 6627, CEP 31270-901\\
Pampulha - Belo Horizonte - Brazil}
\email{fernandez@ufmg.br}

\author[D. Mar\'in]{David Mar\'in}
\address[D. Mar\'in]{Departament de Matem\`atiques, Universitat Aut\`onoma de Barcelona, Edifici C, 08193 Cerdanyola del Vall\`es, Barcelona, Spain}
\email{david.marin@uab.cat}

\author[R. Rosas]{Rudy Rosas}
\address[R. Rosas]{Departamento de Ciencias, Secci\'on Matem\'aticas, Pontificia Universidad Cat\'olica del Per\'u, Av. Universitaria 1801, San Miguel, Lima, Per\'u}
\email{rudy.rosas@pucp.edu.pe}



\thanks{ The first author acknowledges support from CNPq Projeto Universal 408687/2023-1 ``Geometria das Equações Diferenciais Algébricas''
The second author acknowledges support from CNPq Projeto Universal 408687/2023-1 "Geometria das Equa\c{c}\~oes Diferenciais Alg\'ebricas" and CNPq-Brazil PQ-306011/2023-9.
The third author acknowledges support from 
the Ministry of Science, Innovation and Universities of Spain through the grant 
PID2021-125625NB-I00  and from the Agency for Management of University and Research Grants of Catalonia through the grant 2021SGR01015.
The fourth author is partially supported by the Vicerrectorado de Investigación de la Pontificia Universidad Católica del Perú.}

\begin{document}
%\linenumbers

\begin{abstract}
In this paper, we revisit local invariants (G\'omez-Mont-Seade-Verjovsky,  variation, Camacho-Sad and  Baum-Bott indices) associated with singular holomorphic foliations on $(\C^2,0)$ and we provide
semi-global
formulas for them in terms
of the reduction of singularities of the foliation.  A key technical
ingredient is the Cholesky-type factorization of the intersection matrix of the exceptional
divisor, which allows for an explicit control of multiplicities and indices along the resolution
process. Using this factorization, we express the Milnor number and other indices as quadratic
forms in intersection vectors associated to balanced divisors introduced by Y. Genzmer.
As a main application, we address a conjecture posed by A. Szawlowski concerning pencils
of plane holomorphic germs. We prove that the excess of Milnor numbers along the pencil
is precisely captured by the invariants derived from our formulas, thereby confirming the
conjecture in full generality. This also yields a new expression for the dimension of the parameter
space of universal unfoldings of meromorphic functions in the sense of T. Suwa.
\end{abstract}

\maketitle

\tableofcontents
\section{Introduction}

Holomorphic foliations on complex surfaces, particularly on $(\mathbb{C}^2, 0)$, have been
intensively studied due to their deep connections with singularity theory, complex dynamics,
and algebraic geometry. 
Since the pioneering works of Camacho, Lins Neto, and Sad, 
various local invariants have been introduced to describe and classify singularities of 
foliations. Among them, the Milnor number, the Camacho-Sad index, 
the Gómez-Mont-Seade-Verjovsky (GSV) index, the variation index, 
and the Baum-Bott index play a fundamental role.

In this paper, we revisit these invariants and provide \emph{semi-global} formulas for them 
in terms of the numerical data given by the reduction of singularities of the foliation. This perspective not only rederives several known results 
(cf.~\cite{Brunella,FP-GB-SM2021,Arturo}) but also extends them to more general settings. A key technical ingredient is the Cholesky-type factorization of the 
self-intersection matrix of the exceptional divisor, 
which allows for an explicit control of multiplicities and indices along the 
resolution process. Using this factorization, we express the Milnor number and 
other indices as quadratic forms in intersection vectors associated to balanced divisors 
introduced by Genzmer~\cite{Genzmer}.

As a main application, we address a conjecture posed by Szawlowski~\cite{Sa} 
concerning pencils of plane holomorphic germs. Following \cite{Sa}, given a pair of germs $f, g \in \mathcal O_{\C^2,0}$  we define the \textbf{\textit{Milnor number of the pair}} as
$$
\mu (f,g) :=i_0(f_xg-g_xf, f_yg - g_yf).
$$
Let \( \mathcal{F}_{f/g} \) be the germ of the foliation defined by \( \omega = g\,df - f\,dg \). It is clear that if \( 0 \) is an isolated singularity of \( \omega \), then \( \mu(f, g) = \mu_0(\mathcal{F}_{f/g}) \), where \(\mu_0(\mathcal{F}_{f/g}) \) denotes the Milnor number of the foliation \( \mathcal{F}_{f/g} \). 
Assume that \( f \) and \( g \) are non-trivial, coprime germs at the origin of \( \mathbb{C}^2 \). It is well known that the function
\[
\mu : \mathbb{P}^1 \to \mathbb{N}_{\geq 0} \cup \{\infty\}, \quad [\alpha:\beta] \mapsto \mu_0(\alpha f + \beta g)
\]
is upper semicontinuous with respect to the Zariski topology. Therefore, there exists a finite subset \( \mathcal{B}(f,g) \subset \mathbb{P}^1 \), called the \textit{\textbf{bifurcation set}} of the pencil, such that $\mu$ is constant on \( \mathbb{P}^1 \setminus \mathcal{B}(f,g) \). This constant value is referred to as the \textit{\textbf{generic Milnor number}} and is denoted by \( \mu_{\mathrm{gen}}(f,g) \). Moreover, for each element of \( \mathcal{B}(f,g) \), the Milnor number is strictly greater than \( \mu_{\mathrm{gen}}(f,g) \). Szawlowski in \cite[Section 4]{Sa} conjectured that there is a relationship between $\mu(f,g)$ and the ``\textit{total excess}''  of $\mu$ on the bifurcation set. Precisely, he proposed the following \textbf{\textit{Bifurcation Formula}:}
$$
\mu(f,g)=\mu_0(fg)+\sum_{[\alpha:\beta]\in\mathcal{B}(f,g)^{*}}\left(\mu_0(\alpha f +\beta g)-\mu_{gen}(f,g)\right),
$$
where $\mathcal{B}(f,g)^{*}:=\mathcal{B}(f,g)\setminus\{0,\infty\}$.
\par Szawlowski verified the formula in some examples and proved its validity if $f$ or $g$ is smooth \cite[Proposition~4.4]{Sa}. In this work, we provide two different proofs of the Bifurcation formula using the theory of indices of foliations developed in the paper.

This yields a new expression for the dimension of the parameter space of universal unfoldings 
of meromorphic functions in the sense of Suwa~\cite{Suwa} and Bodin-Pichon~\cite{Bodin-Pichon}.

The structure of the paper is as follows: 
In Section~2 we derive formulas for multiplicities and the Milnor number from desingularization data. 
Section~3 is devoted to the computation of several indices of foliations using these tools. 
In Section~4 we apply our results to prove the bifurcation formula for pencils and discuss 
its consequences for semitame functions and unfoldings. 

%%%%%%%%%%%%%%%%%%%%%%%%%%%%%%%%%%%%%%%%%%%%%%%%%%%%%%%%%%%%%%%%%%%%%%%%%%%%%%%%%%%%%%%%%%%%%%%%%%%%%%%%
%\section{Introduction}

%In this paper, we study local invariants associated with singular holomorphic foliations on $\C^2$. In particular, we provide semi-global formulas for the Milnor number of a germ of foliation, as well as for several fundamental indices: the Camacho-Sad index, the G\'omez-Mont-Seade-Verjovsky index, the variation index, and the Baum-Bott index. These expressions establish a connection between local data of the foliation and global information about its reduction of singularities and we reobtain and extend some results involving these indices. 

%As a main application of our formulas, we address the bifurcation conjecture in the context of foliations induced by pencils of curves. We prove that the excess of Milnor numbers along the pencil matches an explicit quantity computed using our invariants, thereby confirming the conjecture in this setting.

%%%%%%%%%%%%%%%%%%%%%%%%%%%%%%%%%%%%%%%%%%%%%%%%%%%%%%%%%%%%%%%%%%%%%%%%%%%%%%%%%%%%%%%%%%%%%%%%%%%%%%%
%%%%%%%%%%%%%%%%%%%%%%%%%%%%%%%%%%%%%%%%%%%%%%%%%%%%%%%%%%%%%%%%%%%%%%%%%%%%%%%%%%%%%%%%%%%%%%%%%%%%%%%

\section{Multiplicities and Milnor number from  desingularization data}
Let $\F$ be a holomorphic foliation on $(\C^2,0)$. A germ divisor $\displaystyle \B=\sum\limits_{B\in \operatorname{Sep}(\mathcal{F})} a_{B}\, B$, $a_B\in\{-1,0,1\}$ is called \textbf{balanced} if it satisfies the following conditions:
\begin{enumerate}[(a)]
\item if $B$ is an isolated separatrix of $\F$ then $a_B=1$, 
\item if $a_B=-1$ then $B$ is a non-isolated (dicritical) separatrix of $\F$,
\item for each non-invariant (dicritical) component $D$ of the exceptional divisor $E$ of the reduction $\pi$ of singularities of $\F$ we have $\sum_{B}a_B \overline{B}\cdot D=2-\mr{val}_D$, where $\mr{val}_D$ is the valence of $D$, i.e. the number of irreducible components of $\overline{E\setminus D}$ meeting $D$, and $\overline{B}$ is the strict transform of $B$.
\end{enumerate}
Notice that a non-dicritical foliation has a unique balanced divisor which is the sum of the isolated separatrix. In contrast, a dicritical foliation have infinitely many balanced divisors.

Let 
\[
\pi = \pi_n \circ \cdots \circ \pi_1 : (M, E) \longrightarrow (\mathbb{C}^2, 0)
\]
be a composition of $n$ blow-ups. In what follows we define a combinatorial data associated to $\F$ with respect to $\pi$.  
Denote by $E_1, \dots, E_n$ the irreducible components of the exceptional divisor 
\[
E = \pi^{-1}(0),
\]
and let $A = (A_{ij})$ be the self-intersection matrix of $E$, where 
\[
A_{ij} = E_i \cdot E_j, \qquad i,j = 1, \dots, n.
\]
Moreover, associated with $\pi$ we define a sequence of matrices $A_1, A_2, \dots, A_n = A$, where, for each $j$, the matrix $A_j$ denotes the self-intersection matrix of the exceptional divisor of $\pi_j \circ \cdots \circ \pi_1.$ We also consider the following column vectors:
\[
S_{\mathcal{B}} = 
\begin{pmatrix}
\overline{\mathcal{B}} \cdot E_1 \\
\vdots \\
\overline{\mathcal{B}} \cdot E_n
\end{pmatrix}, 
\qquad
u = 
\begin{pmatrix}
1 \\
\vdots \\
1
\end{pmatrix},
\qquad
\iota = 
\begin{pmatrix}
\iota_1 \\
\vdots \\
\iota_n
\end{pmatrix},
\]
where $\overline{\mathcal{B}}$ denotes the strict transform of $\mathcal{B}$ by $\pi$, and 
\[
\iota_j = 
\begin{cases}
1, & \text{if } E_j \text{ is } \pi^* \mathcal{F}\text{-invariant}, \\
0, & \text{otherwise}.
\end{cases}
\] 
and let  $\delta = u-\iota$ be the vector corresponding to the dicritical components. We define a sequence of matrices $F_1, F_2, \dots, F_n$ associated with the sequence of blow-ups.  
Each matrix $F_k$ is a $k \times k$ lower triangular matrix with $1$'s along the diagonal.  
We start with $F_1 = (1),$ and for $k \geq 2$, we define
\[
F_{k} =
\begin{pmatrix}
F_{k-1} & 0 \\
- e_k & 1
\end{pmatrix},
\]
where the row vector $e_k = (e_{k,1}, \dots, e_{k,k-1})$ is given by
\[
e_{k,j} =
\begin{cases}
1, & \text{if the $k$-th blow-up point lies on the divisor } E_j, \\
0, & \text{otherwise}.
\end{cases}
\]
Denote by $F=F_n$. The matrix $F$ is given explicitly by
\[
F = 
\begin{pmatrix}
1 & 0 & 0 & \cdots & 0 \\
-1 & 1 & 0 & \cdots & 0 \\
\vdots & \vdots & \ddots & \ddots & \vdots \\
- E_{n-1}^{\,n-1} \!\cdot\! E_1^{\,n-1} & - E_{n-1}^{\,n-1} \!\cdot\! E_2^{\,n-1} & \cdots & 1 & 0 \\
- E_n^{\,n} \!\cdot\! E_1^{\,n} & - E_n^{\,n} \!\cdot\! E_2^{\,n} & \cdots & - E_n^{\,n} \!\cdot\! E_{n-1}^{\,n} & 1
\end{pmatrix},
\]
where $E_1^i, \dots, E_{i-1}^i$ denote the strict transforms of the components of 
\[
(\pi_{i-1} \circ \cdots \circ \pi_1)^{-1}(0) 
= E_1^{\,i-1} \cup \cdots \cup E_{i-1}^{\,i-1},
\]
and $E_i^i$ denotes the exceptional divisor of $\pi_i$, for $i=1,\dots,n$. We will need some auxiliary results.
 
\begin{lemma} \label{Cholesky}
The self-intersection matrix $A$ of the exceptional divisor $E$ admits the Cholesky-type factorization
\[
A = - F^{\mathsf{T}} F,
\]
where $F$ is the lower triangular matrix defined above.
\end{lemma}

\begin{proof}
By induction on $n$. The case $n=1$ is obvious. Assume now that the result holds for $n-1$ and let us prove it for $n$. We have that

 \begin{equation}\label{Fn}
 F_n=\left(\begin{array}{cc}F_{n-1} & 0\\ -e_n & 1\end{array}\right)\quad\text{and}\quad F_n^{-1}=\left(\begin{array}{cc}F_{n-1}^{-1} &0\\[1mm] e_nF_{n-1}^{-1} & 1\end{array}\right),
 \end{equation}
where $e_n=(E_n\cdot E_1,\ldots,E_n\cdot E_{n-1})$. 
Then
\[-F_n^{\mathsf{T}}F_n=\left(\begin{array}{cc}-F_{n-1}^{\mathsf{T}}& e_n^{\mathsf{T}}\\ 0 & -1 \end{array}\right)\left(\begin{array}{cc} F_{n-1} & 0\\ -e_n & 1\end{array}\right)=\left(\begin{array}{cc}-F_{n-1}^{\mathsf{T}}F_{n-1}-e_n^{\mathsf{T}}e_n & e_n^{\mathsf{T}}\\ e_n & -1 \end{array}\right)=A_n\]
by the inductive hypothesis $-F_{n-1}^{\mathsf{T}}F_{n-1}=A_{n-1}$ and the fact 
that the  entries of the matrix $e_n^{\mathsf{T}}e_n$ are $1$ in the positions $(i,i)$, $(i,j)$, $(j,i)$ and $(j,j)$ and $0$ otherwise, where we are assuming that the center of the last blow-up $\pi_n$ is $E_i^{n-1}\cap E_j^{n-1}$. If the last center is not a corner the computation is similar.
\end{proof}

\begin{lemma}\label{aux1}
Let  $L$ and $B$ be curves through the origin with algebraic multiplicity $1$ and $a$ respectively. 
Let $\pi$ be the blow-up of the origin, $E=\pi^{-1}(0)$ and $\bar L$ and $\bar B$  the strict transforms of $L$ and $B$ respectively. Then $\pi^*B=\bar B+aE$, $\pi^*L=\bar L+E$, $\bar B\cdot E=a$ and $L\cdot B-a=\bar L\cdot\bar B$.
\end{lemma}
\begin{proof} We have that $L\cdot B=\pi^*L\cdot\pi^* B=(\bar L+E)\cdot(\bar B+aE)=\bar L\cdot \bar B+a\bar L\cdot E+\bar B\cdot L+aE\cdot E=\bar L\cdot \bar B-a$ because $0=\pi^*L\cdot E=\bar L\cdot E-1$ and $0=\pi^*B\cdot E=\bar B\cdot E-a$.
\end{proof}

\begin{lemma}\label{aux} 
 $F^{\mathsf{T}}u=2u+(A_{11},\ldots,A_{nn})^{\mathsf{T}}$.
 \end{lemma}
 \begin{proof}
It follows by induction on $n$ using (\ref{Fn}). The case $n=1$ being obvious, let us consider the inductive step:
 \begin{align*}
(u^n)^{\mathsf{T}}F_n&=((u^{n-1})^{\mathsf{T}},1)\left(\begin{array}{cc}
 F_{n-1} & 0\\ -e_n & 1\end{array}\right)=((u^{n-1})^{\mathsf{T}}F_{n-1}-e_n,1)\\
 &=(2+A_{11}^{n-1}-e_{n,1},\ldots,2+A_{n-1,n-1}^{n-1}-e_{n,n-1},2-1)\\
 &=(2+A_{11}^n,\ldots,2+A_{n-1,n-1}^n,2+A_{nn}^n).
 \end{align*}
 \end{proof}
 
\begin{lemma}\label{bal}
If $\B$ is balanced then $\langle \delta,S_{\B}+(A-F^{\mathsf{T}})u\rangle=0$.
\end{lemma}

\begin{proof}
Notice that if $\varepsilon_i$ is the $i$-th element of the canonical basis of $\Z^n$ we have that the valence of $E_i$ is
\[\mr{val}_{E_i}=A_{1i}+\cdots+A_{i-1,i}+A_{i+1,i}+\cdots+A_{ni}=\langle A\varepsilon_i,u\rangle-A_{ii}=\langle A\varepsilon_i,u\rangle -\langle A\varepsilon_i,\varepsilon_i\rangle=\langle A\varepsilon_i,u-\varepsilon_i\rangle.\]
Condition (c) for a balanced divisor $\B$ is equivalent to $\delta_i(S_{\B,i}-2+\langle A\varepsilon_i,u-\varepsilon_i\rangle)=0$ for each $i=1,\ldots,n$.
Summing up all these equalities we obtain
\[0=\langle \delta,S_{\B}\rangle+\sum_{i=1}^n\delta_i(-2+\langle A \varepsilon_i,u-\varepsilon_i\rangle)=\langle\delta, S_{\B}\rangle+\sum_{i=1}^n\delta_i(\langle \varepsilon_i,Au\rangle-(A_{ii}+2))=\langle\delta,S_{\B}+Au-F^{\mathsf{T}}u\rangle,\]
thanks to Lemma~\ref{aux}, asserting that $A_{ii}+2=(F^{\mathsf{T}}u)_i=\langle e_i, F^{\mathsf{T}}u\rangle$.
\end{proof}
Recall that a foliation is said to be of \textbf{second class} if none of the singularities of $\pi^*\mathcal{F}$ are tangent saddle-nodes; that is, whenever saddle-nodes appear in the reduction, their strong separatrix is transverse to the exceptional divisor. If no saddle-nodes appear in the reduction we say that $\F$ is a \textbf{generalized curve}.
The balanced divisors characterize second class foliations:
\begin{proposition}[\cite{Genzmer}, Proposition 2.6]\label{Genzmer}
Let $\F$ be a foliation germ on $(\C^2,0)$ with balanced divisor $\B$. Then $\F$ is of second class if and only if $\nu_0(\F)=\nu_0(\B)-1$.
\end{proposition}

Let us define  the vector  $\ell=(\ell_1,\ldots,\ell_n)^{\mathsf{T}}$  of multiplicities (or discrepancies) of $\pi^*\F$ along each component of $E$ as follows. If $\pi_j$ is the blow up of a point $p_{j-1}$  and $\omega_j$ is a $1$-form defining the foliation around $p_{j-1}$ then $\ell_j$ is the vanishing order of $\pi_j^*(\omega_j)$ along $E_j$. In fact, it is easy to see that $\ell_j=\nu_{p_{j-1}}((\pi_{j-1}\circ\cdots\circ\pi_1)^*\F)+1-\iota_j$. \\

In this context, Genzmer's Proposition give us an expression for $\ell_1$ as $\ell_1 = \nu_0(\B)- \iota_1$. The following theorem %generalize 
is a generalization of 
this formula.

\begin{theorem}\label{ell}
Let $\F$ be  a second class foliation on $(\C^2,0)$ with balanced divisor $\B$. Then
$$\ell=(F^{-1})^{\mathsf{T}}S_{\B}-F\iota $$
\end{theorem}

\begin{proof}
We proceed by induction on $n$. 
The case $n=1$ follows, using that $\overline{\B}\cdot E_1=\nu_0(\B)$ and $\ell_1=\nu_0(\F)+1-\iota_1$, from Proposition~\ref{Genzmer}.

 Assume now that the result holds for $n-1$ and let us prove it for $n$.
 
If the center of the last blow-up is a corner $p$, up to reordering the components of $E$, we can assume that $e_n=(0,\ldots,0,1,1)$.
We claim that
 \begin{equation}\label{SBn}
 S_{\B}^n=\left(\begin{array}{c}S_{\B}^{n-1}\\ 0\end{array}\right)+a\left(\begin{array}{c}0\\ \vdots\\ 0 \\ -1 \\ -1 \\ 1\end{array}\right),
 \end{equation}
 where $a=\nu_p(\overline{\B}_{n-1})$ is the algebraic multiplicity of the strict transform of $\B$ by $\pi_{n-1}\circ\cdots\circ\pi_1$ at $p$. The equalities of the coordinates other than $n-2$ and $n-1$ are clear.
To see that $\overline{\B}_n\cdot E_{n-i}^n=\overline{\B}_{n-1}\cdot E_{n-i}^{n-1}-a$ for $i=1,2$ we apply Lemma~\ref{aux1} to $L=E_{n-i}$ and $B=\overline{\B}_{n-1}$.
Using (\ref{SBn}) and the inductive hypothesis we deduce that
  \begin{align*}
 (F_n^{-1})^{\mathsf{T}}S_{\B}^n-F_n\iota^n=&\left(\begin{array}{cc}
 (F_{n-1}^{-1})^{\mathsf{T}}& (F_{n-1}^{-1})^{\mathsf{T}}e_n^{\mathsf{T}}\\ 0 & 1
 \end{array}\right)\left(\begin{array}{c}S_{\B}^{n-1}\\ 0\end{array}\right)
 \\&
 +a
 \left(\begin{array}{cc}
 (F_{n-1}^{-1})^{\mathsf{T}} & (F_{n-1}^{-1})^{\mathsf{T}}e_n^{\mathsf{T}}\\ 0 & 1
 \end{array}\right)\left(\begin{array}{c}0\\ \vdots\\ 0 \\ -1 \\ -1 \\ 1\end{array}\right)
 %\\&
 -\left(\begin{array}{cc}F_{n-1} & 0\\ -e_n & 1\end{array}\right)\left(\begin{array}{c} \iota^{n-1}\\ \iota^n_n\end{array}\right)\\
 =&\left(\begin{array}{c} (F_{n-1}^{-1})^{\mathsf{T}}S_{\B}^{n-1}-F_{n-1}\iota^{n-1}\\ \iota_{n-2}^{n-1}+\iota_{n-1}^{n-1}-\iota_n^n\end{array}\right)+\left(\begin{array}{c}0\\ \vdots\\ 0\\ a\end{array}\right)=\left(\begin{array}{c}\ell_1\\ \vdots\\ \ell_{n-1}\\ \ell'_n\end{array}\right).
 \end{align*}
Let us see  that $\ell_n':=a+\iota_{n-2}+\iota_{n-1}-\iota_n$ coincides with $\ell_n$. 
 Since $\overline{\B}_{n-1}+\iota_{n-2}E_{n-2}^{n-1}+\iota_{n-1}E_{n-1}^{n-1}$ is  a balanced divisor for the germ of the foliation $\F_{n-1}:=(\pi_{n-1}\circ\cdots\circ\pi_1)^*\F$ at $p$, by applying again Genzmer's formula in Proposition~\ref{Genzmer} we obtain that 
 \[\ell_n=\nu_p(\F_{n-1})+1-\iota_n=\nu_p(\overline{\B}_{n-1}+\iota_{n-2}E_{n-2}^{n-1}+\iota_{n-1}E_{n-1}^{n-1})-\iota_n=a+\iota_{n-2}+\iota_{n-1}-\iota_n=\ell_n'.\]
 
In case where the center $p$ of the last blow-up is a regular point of $E$  we can assume that $e_n=(0,\ldots,0,1)$ and by applying again Lemma~\ref{aux} we obtain that
 \[S_{\B}^n=\left(\begin{array}{c}S_{\B}^{n-1}\\ 0\end{array}\right)+a\left(\begin{array}{c}0\\ \vdots\\ 0  \\ -1 \\ 1\end{array}\right).\] 
 The computations made in the case of a corner can be easily adapted to this simpler situation.
 \end{proof}

\begin{remark}
Using Seifert-Van Kampen theorem, if $\pi$ is a desingularization
of a reduced germ of curve $C:f_1\cdots f_m=0$ then 
the fundamental group of the complement of $C$ in a Milnor ball $\mb B$ is
generated by loops $\gamma_K$, where $ K$ varies in the set of irreducible components of $\pi^{-1}(C)=\bar C_1\cup\cdots\cup\bar C_m\cup E_1\cup\cdots E_n$, with the following relations
\[\prod_K\gamma_K^{K\cdot E_i}=\gamma_{E_i}^{E_i\cdot E_i}\prod_{K\neq E_i} \gamma_{K}^{K\cdot E_i}=1,\quad i=1,\ldots,n.\]
Then $H_1(\mb B\setminus C,\Z)=\bigoplus_{K}\Z c_K/\langle \sum_K K\cdot E_i c_K,i=1,\ldots,n\rangle=\Z c_1\oplus\cdots\oplus\Z c_m$, where $c_i=c_{\bar C_i}$ is a cycle of index $1$ around $f_i=0$ and $0$ around $f_j=0$ if $j\neq i$.
Thus 
\[(c_{E_1},\ldots, c_{E_n})A=-\sum_{i=1}^m(\bar C_i\cdot E_1,\ldots,\bar C_i\cdot E_n)c_i\]
and
\[(c_{E_1},\ldots, c_{E_n})=-\sum_{i=1}^m(\bar C_i\cdot E_1,\ldots,\bar C_i\cdot E_n)A^{-1}c_i.\]
Since $\frac{1}{2\pi\sqrt{-1}}\int_{c_i}\frac{d\pi^*f_i}{\pi^* f_i}=\frac{1}{2\pi\sqrt{-1}}\int_{\pi(c_i)}\frac{df_i}{f_i}=1$,
the  vector of vanishing orders of $\pi^*f_i$ along each $E_j$ is 
\[M_{f_i}=\left(\frac{1}{2\pi\sqrt{-1}}\int_{c_{E_j}}\frac{d\pi^*f_i}{\pi^*f_i}\right)_j=-A^{-1}(\bar C_i\cdot E_j)_j=-A^{-1}S_{C_i}.\]
By linearity, $M_f=M_{f_1\cdots f_m}=\sum_{i=1}^m-A^{-1}S_{C_i}=-A^{-1}S_C$.
By applying Theorem~\ref{ell} to the foliation with holomorphic first integral $f$ we obtain that
\[m_f:=(\nu_{p_{j-1}}(\pi_{j-1}\circ\cdots\circ\pi_1)^*df)_j=F(M_f-u).\]
\end{remark}

Recall that the \textbf{Milnor number} $\mu_0(\F)$ of the foliation $\F$ at $0\in\C^2$ given by the $1$-form $\omega=P(x,y)dx+Q(x,y)dy$ is defined by 
\[ \mu_0(\F)=i_0(P,Q),\] 
where $i_0(P,Q)$ denotes the intersection number of two germs $P$ and $Q$ at the origin. Remember that we consider  $P$ and $Q$  coprime, so $\mu_0(\F)$ is a non negative integer. In \cite[Theorem A]{CLS} it was proved that the Milnor number of a foliation is  a topological invariant. As a consequence of the previous result we obtain a formula from the reduction data.

\begin{theorem}\label{muF}
If $\pi$ is a reduction of singularities, not necessarily minimal, of a generalized curve $\F$ on $(\C^2,0)$ with balanced divisor $\B$ then
\[\mu_0(\F)=\langle -A^{-1}S_{\B},S_{\B}\rangle-\langle S_{\B},(I+F^{-1})u\rangle+1
,\]
\end{theorem}
\begin{proof} The Van den Essen formula implies that 
\[\mu_0(\F)=\sharp\mr{Sing}(\pi^*\F)+N(\ell),\] 
where $N(\ell)=\sum_{j=1}^n(\ell_j^2-\ell_j-1)=\langle\ell,\ell\rangle-\langle\ell,u\rangle-\langle u,u\rangle$. We can write $\B=\mc I+\mc D$ where $\mc I$ are the isolated separatrices of $\F$ and the support of $\mc D$ consists of some dicritical separatrices of $\F$. The number of singularities of $\pi^*\F$ is 
\[ 
\underbrace{
\langle S_{\mc I},u\rangle
}_{\begin{array}{c}\scriptstyle\text{attaching points of}\\ \scriptstyle\text{isolated separatrices}\end{array}}
+\underbrace{\langle u,u\rangle-1}_{\text{corners}}-\underbrace{\sum_{D\text{ dic}}\mr{val}_D}_{\begin{array}{c}\scriptstyle\text{attaching points of}\\ \scriptstyle\text{dicritical components}\end{array}}\] 
and, using that $\mc B$ is balanced, we have $\sum_{D\text{ dic}}\mr{val}_D=%\langle A\delta,u-\delta\rangle=
2\langle\delta,u\rangle-\langle S_{\mc B},\delta\rangle$, where $\delta=u-\iota$ is the vector of dicritical components of~$E$.
Since  $S_{\mc B}=S_{\mc I}+S_{\mc D}$ and $\langle S_{\mc I},\delta\rangle=\langle S_{\mc D},\iota\rangle=0$ we deduce that $\langle S_{\mc I},u\rangle=\langle S_{\mc B},u-\delta\rangle$.
Thus we obtain 
\begin{align*}
\mu_0(\F)&=\langle S_{\mc B},u-\delta\rangle+\langle u,u\rangle-1-\Big(2\langle\delta,u\rangle-\langle S_{\mc B},\delta\rangle\Big)+\langle\ell,\ell\rangle-\langle \ell,u\rangle-\langle u,u\rangle\\
&=\langle S_{\mc B},u\rangle-2\langle \delta,u\rangle+\langle\ell,\ell\rangle-\langle\ell,u\rangle-1.
\end{align*}
Since, according to Theorem~\ref{ell},
\begin{align*}
\langle\ell,\ell\rangle&=\langle (F^{-1})^{\mathsf{T}}S_{\mc B},(F^{-1})^{\mathsf{T}}S_{\mc B}\rangle-2\langle (F^{-1})^{\mathsf{T}}S_{\mc B},F\iota\rangle+\langle F\iota,F\iota\rangle=\langle -A^{-1}S_{\mc B},S_{\mc B}\rangle-2\langle S_{\mc B},\iota\rangle-\langle \iota,A\iota\rangle
\end{align*}
and
\[\langle \ell,u\rangle=\langle (F^{-1})^{\mathsf{T}}\big(S_{\mc B}-F^{\mathsf{T}}F\iota\big),u\rangle=\langle S_{\mc B} + A\iota,F^{-1}u\rangle,\]
we deduce that
\begin{align*}
\mu_0(\F)&=\langle -A^{-1}S_{\mc B},S_{\mc B}\rangle+\langle S_{\mc B},u\rangle-2\langle \delta,u\rangle-2\langle S_{\mc B},u-\delta\rangle-\langle \iota, A\iota\rangle-\langle S_{\mc B}+A\iota,F^{-1}u\rangle-1\\
&=\langle -A^{-1}S_{\mc B},S_{\mc B}\rangle+\langle S_{\mc B},u-2u-F^{-1}u\rangle+2\langle S_{\mc B},\delta\rangle{\color{red}-}\langle A\iota,F^{-1}u\rangle-2\langle\delta,u\rangle-\langle\iota,A\iota\rangle-1\\
&=\langle -A^{-1}S_{\mc B},S_{\mc B}\rangle-\langle S_{\mc B},(I+F^{-1})u\rangle-1+\beta,
\end{align*}
where $\beta=2\langle S_{\mc B},\delta\rangle-\langle A\iota,F^{-1}u\rangle-2\langle\delta,u\rangle-\langle\iota,A\iota\rangle$. It remains to prove that $\beta=2$.
By applying Lemma~\ref{bal} in the first equality, the relation  $-A=F^{\mathsf{T}}F$ in the second one and Lemma~\ref{aux} in the fourth one, we obtain:
\begin{align*}
\beta&=2\langle(F^{\mathsf{T}}-A)u,\delta\rangle-2\langle u,\delta\rangle-\langle A\iota,F^{-1}u+\iota\rangle
=\langle 2F^{\mathsf{T}}u-2u-2Au,\delta\rangle+\langle u-\delta,F^{\mathsf{T}}u-Au+A\delta\rangle\\
&=\langle F^{\mathsf{T}}u-2u-Au-A\delta,\delta\rangle+\langle u,F^{\mathsf{T}}u-Au+A\delta\rangle
=\langle \alpha-Au-A\delta,\delta\rangle+\langle u,\alpha+2u-Au+A\delta\rangle\\
&=\langle\alpha-A\delta,\delta\rangle+\langle u,\alpha+2u-Au\rangle,
\end{align*}
where $\alpha=(A_{11},\ldots,A_{nn})^{\mathsf{T}}$.
Using that two dicritical components do not meet we have that
\[\langle A\delta,\delta\rangle=\langle \sum_{E_i\text{ dic}}A\varepsilon_i,\sum_{E_j\text{ dic}}\varepsilon_j\rangle=\sum_{E_i,E_j \text{ dic}}A_{ij}=\delta_1A_{11}+\cdots+\delta_nA_{nn}=\langle\delta,\alpha\rangle.\]
Finally, using that the dual graph of $E$ is a tree (so that its Euler characteristic is one) we have that its number of corners $\sum_{i<j}A_{ij}$ of $E$ is $n-1$ and consequently
\[\langle u,\alpha+2u-Au\rangle=2n-\sum_{i,j}A_{ij}+\sum_{i}A_{ii}=2n-2\sum_{i<j}A_{ij}=2.\]
This completes the proof.
\end{proof}

\begin{remark}
Let ${f = 0}$ be a reduced curve, and let $\mathcal{F}$ be the (non-dicritical) foliation defined by $df = 0$. In this case, $\mathcal{F}$ and $f$ share the same resolution, $\mu_0(\mathcal{F}) = \mu_0(f)$, and the balanced divisor of $\mathcal{F}$ is $\mathcal{B} = \{f = 0\}$. We denote $S_{\mathcal{B}}$ by $S_f$ in this situation.
\end{remark}

\begin{corollary}\label{cor}
Let ${f = 0}$ be a reduced curve in $(\C^2,0)$. Let $A$, $F$, $S_f$ be the previous matrices and vector corresponding to a resolution of $f$. Then $\mu_0(f)=\langle -A^{-1}S_f, S_f\rangle -\langle (I+F^{-1})u, S_f \rangle +1$.
\end{corollary}
\begin{proof}
The proof is a direct application of the last theorem.
\end{proof}

\begin{corollary}\label{cor1}
Let $\{f=0\}$ and $\{g=0\}$ be coprime and reduced curves in $(\C^2,0)$. Let $A, F, S_f, S_g$ be the data corresponding to a resolution of $\{f g=0\}$, then 
$i_0(f,g) = \langle -A^{-1}S_f, S_g\rangle.$
\end{corollary}

\begin{proof}
It is well-known that $\mu_0(fg)=\mu_0(f)+\mu_0(g)+2i_0(f,g)-1$ and, since $S_{fg}=S_f+S_g$, the result is a consequence of the previous corollary. 
\end{proof}

%%%%%%%%%%%%%%%%%%%%%%%%%%%%%%%%%%%%%%%%%%%%%%%%%%%%%%%%%%%%%%%%%%%%%%%%%%%%%%
%%%%%%%%%%%%%%%%%%%%%%%%%%%%%%%%%%%%%%%%%%%%%%%%%%%%%%%%%%%%%%%%%%%%%%%%%%%%%%

\section{Indices of foliations}

Throughout this section, we will always work under the assumption that \( \mathcal{F} : \omega = 0\) is a \textbf{generalized curve} and let $\mc B$ be a balanced divisor for $\F$. Let $\pi=\pi_m\circ\cdots\circ\pi_1$ be a resolution of $\F$ with intersection matrix $A=-F^{\mathsf{T}}F$, Cholesky matrix $F$ and invariant vector $\iota$. A useful tool for obtaining our formulas is given by the following lemma.
\begin{lemma}\label{multiplicidades de curva}
Let $C$ be an invariant curve of $\F$, then the vector of algebraic multiplicities of the strict transforms of $C$ is given by
$$
(\nu_0(C),\nu_{p_1}(\bar C_1),\ldots, \nu_{p_{m-1}}(\bar C_{m-1}))^{\mathsf{T}}=(F^{-1})^{\mathsf{T}}S_C.
$$
\end{lemma}
\begin{proof}
Assume that $\pi=\pi_m\circ\cdots\circ\pi_n\circ\cdots\circ\pi_1$ is a desingularization of the irreducible $\F$-invariant curve $C$, that the strict transform  $\bar C_i$  meet $E_i^i$ at $p_i$ for $i=1,\ldots,n$  and 
$\bar C_{i}\cap E_i^i=\emptyset$ if $i>n$. Since, recall (\ref{SBn}),
\[S_C^n=\left(\begin{array}{c}S_C^{n-1}\\ 0\end{array}\right)-\nu_{p_{n-1}}(\overline{C}_{n-1})\left(\begin{array}{c}
E_1\cdot E_n\\ \vdots\\ E_n\cdot E_n\end{array}\right),\] 
it follows easily by induction on $n$ that
$(\nu_0(C),\nu_{p_1}(\bar C_1),\ldots, \nu_{p_{n-1}}(\bar C_{n-1}))^{\mathsf{T}}=(F^{-1}_n)^{\mathsf{T}}S_C^n$.
Since
\[S_C^m=\left(\begin{array}{c}S^n_C\\ 0\end{array}\right)\]
we conclude that
\begin{align*}
(\nu_0(C),\nu_{p_1}(\bar C_1),\ldots, \nu_{p_{m-1}}(\bar C_{m-1}))^{\mathsf{T}}&=(\nu_0(C),\nu_{p_1}(\bar C_1),\ldots, \nu_{p_{n-1}}(\bar C_{n-1}),0,\ldots,0)^{\mathsf{T}}\\
&=
(F^{-1}_m)^{\mathsf{T}}S_C^m.
\end{align*}
\end{proof}

\subsection{G\'omez-Mont - Seade -Verjovsky index}
 Let $C=\{f=0\}$ be a reduced effective divisor invariant by $\F$, then we can write (\cite{Su})
\begin{equation}\label{descomposicion}
g\omega = hdf + f \eta,
 \end{equation}
with $f$ and $h$ and $f$ and $h$ relatively prime and $\eta$ a holomorphic $1$--form. We define  the \textbf{G\'omez-Mont- Seade - Verjovsky index} (GSV index) of $\mathcal{F}$ with respect to $C$ as 
\begin{equation*}
GSV_0(\mathcal{F}, C)= ord_0 \left.\left( \frac{h}{g}\right)\right|_C=\frac{1}{2\pi i}\int_{\partial C}\frac{g}{h}d\left(\frac{h}{g} \right),
\end{equation*} 
where $\partial C$ is the intersection of $C$ with a small sphere around $0$, with the induced orientation.

If $C_1$ and $C_2$ are $\F$-invariant curves without common components, then the following formula holds (cf. \cite{Brunella})
\begin{equation}\label{GSV en suma}
GSV_0(\mathcal{F}, C_1+C_2)=GSV_0(\mathcal{F}, C_1)+GSV_0(\mathcal{F}, C_2)- 2 i_0(C_1,C_2).
\end{equation}
For this index we have the following.
\begin{theorem}\label{GSV}
Let $C$ be a $\F$-invariant reduced effective divisor and let $\B$ any balanced divisor of $\F$, then 
\[
GSV_0(\F,C)=\langle -A^{-1}(S_{\mc B}-S_C),S_C\rangle.\]
\end{theorem}
\begin{proof}
Assume first that $C$ is irreducible. It is known, cf. \cite{Brunella}, that if $\bar\F$ and $\bar C$ are the strict transforms of $\F$ and $C$ by a single blow-up then
\[GSV_0(\F,C)=GSV_{p}(\bar\F,\bar C)+\nu_0(C)(\ell_1-\nu_0(C)),\] 
where $p$ is the intersection point of $\bar C$ with the exceptional divisor.
Continuing with the process of resolving singularities, we obtain
\[GSV_0(\F,C)=\underbrace{GSV_{p_n}(\bar\F_n,\bar C_n)}_{\iota_n}+\sum_{i=1}^{n}
\nu_{p_{i-1}}(\bar C_{i-1})(\ell_i-\nu_{p_{i-1}}(\bar C_{i-1})).\]
Using Lemma \ref{multiplicidades de curva} and Theorem \ref{ell} we obtain
\begin{align*}
\sum_{i=1}^{n}
\nu_{p_{i-1}}(\bar C_{i-1})(\ell_i-\nu_{p_{i-1}}(\bar C_{i-1}))&=\sum_{i=1}^{m}
\nu_{p_{i-1}}(\bar C_{i-1})(\ell_i-\nu_{p_{i-1}}(\bar C_{i-1}))
\\
&=\langle (F^{-1})^{\mathsf{T}}S_C,\ell-(F^{-1})^{\mathsf{T}}S_C\rangle\\
&=\langle (F^{-1})^{\mathsf{T}}S_C,(F^{-1})^{\mathsf{T}}S_{\mc B}-F\iota-(F^{-1})^{\mathsf{T}}S_C\rangle\\
&=\langle S_C,F^{-1}(F^{-1})^{\mathsf{T}}S_{\mc B}-\iota-F^{-1}(F^{-1})^{\mathsf{T}}S_C\rangle\\
&=\langle S_C,-A^{-1}(S_{\mc B}-S_C)-\iota\rangle\\
&=\langle S_C,-A^{-1}(S_{\mc B}-S_C)\rangle-\underbrace{\langle S_C,\iota\rangle}_{\iota_n}
\end{align*}
and the formula for the GSV index in the statement holds in this case. Now, 

\begin{align*}
GSV_0(\F,C_1+C_2)&=GSV_0(\F,C_1)+GSV_0(\F,C_2)-2i_0(C_1,C_2)\\
&=\langle -A^{-1}(S_{\mc B}-S_{C_1}),S_{C_1}\rangle+\langle -A^{-1}(S_{\mc B}-S_{C_2}),S_{C_2}\rangle-2\langle -A^{-1}S_{C_1},S_{C_2}\rangle\\
&=\langle -A^{-1}(S_{\mc B}-S_{C_1}-S_{C_2}),S_{C_1}+S_{C_2}\rangle,
\end{align*}
so that by induction the formula holds for any invariant reduced effective divisor. 
\end{proof}

As an immediate consequence, we recover Proposition $7$ of \cite{Brunella}.
\begin{proposition}
If $\F$ is a non-dicritical generalized curve and $C$ is the union of the separatrices then $GSV_0(\F,C)=0$.
\end{proposition}
\begin{proof}
We just apply the previous theorem observing that $S_{\mc B}=S_C$.
\end{proof}

\subsection{Milnor number along a separatrix} 

Let $C$ be a separatrix of $\F$ with primitive pa\-ra\-me\-tri\-za\-tion $\gamma: (\C,0)\to(\C^2,0)$ and $v$ a vector field defining $\F$, Camacho-Lins Neto-Sad \cite[Section 4]{CLS} defined 
the \textbf{multiplicity of $\F$ along $C$ at $0$} as
$\mu_0(\F,C):=\ord_t \theta(t)$,
where $\theta(t)$ is the unique vector field at $(\C,0)$ such that 
$\gamma_{*} \theta(t)=v\circ\gamma(t)$. 
If $\omega=P(x,y)dx+Q(x,y)dy$ is a 1-form inducing $\F$ and $\gamma(t)=(x(t),y(t))$, we have
\begin{equation*}
\theta(t)=
\begin{cases}
-\frac{Q(\gamma(t))}{x'(t)} & \text{if $x'(t)\neq 0$}
\medskip \\
 \frac{P(\gamma(t))}{y'(t)} & \text{if $y'(t)\neq 0$}.
\end{cases}
\end{equation*}

Following \cite[Section 2]{FP-GB-SM2024}, we define the multiplicity of  $\mathcal{F}$ along any nonempty divisor of separatrices 
$\mathcal{B}=\sum_{C} a_C\cdot C$ of separatrices of $\mathcal F$ at $0$ as follows:
\begin{equation}\label{eq:gmul}
\mu_0(\mathcal{F},\mathcal{B})=\left(\displaystyle\sum_{C}a_{C}\cdot \mu_0(\mathcal{F},C)\right)-\sum_C a_C +1.
\end{equation}
Note that this is equivalent to extend linearly the function $C \mapsto \mu_0(\mathcal{F},C)-1$. 
\begin{remark}
If $\B_1$ and $\B_2$ are divisors of separatrices, it can be verified that $\mu_0(\mathcal{F},\mathcal{B}_1+\B_2)=\mu_0(\mathcal{F},\mathcal{B}_1)+\mu_0(\mathcal{F},\mathcal{B}_2)-1$. From this we can show that, given a formal sum of divisors of separatrices $\sum_j a_j \B_j$, we have that 
\begin{equation}\label{eq:Milnor along a divisor}
\mu_0(\mathcal{F},\sum_j\mathcal{B}_j)=\left(\displaystyle\sum_ja_{j}\cdot \mu_0(\mathcal{F},\B_j)\right)-\sum_j a_j +1.
\end{equation}
\end{remark}
\begin{remark}\label{GSV vs Milnor}
It can be seen that, for $C$ an invariant curve, $\mu_0(\F,C)=GSV_0(\mathcal{F}, C) + \mu_0(C)$. In fact, it follows easily from the definition if $C$ is irreducible, and from the behavior of the indices with respect to the sum of curves in the general case, see for instance \cite[Proposition~4.1]{Arturo}.
\end{remark}

\begin{theorem}\label{IntMilnor}
Let $C$ be a reduced $\F$-invariant divisor and let $\B$ be any balanced divisor, then 
\[
\mu_0(\F,C)=\langle -A^{-1}S_{\mc B}-(I+F^{-1})u,S_C\rangle +1 .\]
\end{theorem}
\begin{proof}
If $C$ is an reduced effective divisor we have (Theorem \ref{GSV} and Corollary \ref{cor})
\begin{align*}
\mu_0(\F,C)&=GSV_0(\mathcal{F}, C) + \mu_0(C) \\
&= \langle -A^{-1}(S_{\mc B}-S_C),S_C\rangle + \langle -A^{-1}S_C, S_C\rangle -\langle (I+F^{-1})u, S_C \rangle +1\\
&=\langle -A^{-1}S_{\mc B}-(I+F^{-1})u,S_C\rangle +1 .
\end{align*}
Finally, we use the fact that both $\mu_0(\mathcal{F}, C) - 1$ and $\langle -A^{-1}S_{\mathcal{B}} - (I + F^{-1})u, S_C \rangle$ are linear in $C$ to state our result for a reduced divisor with polar part.
\end{proof}

Combining this result with Theorem \ref{muF} we obtain
\begin{proposition}{\cite[Proposition 4.7]{FP-GB-SM2021}}\label{mil = intmil}
Let $\F$ be a generalized curve with balanced divisor $\B$, then $\mu_0(\F)=\mu_0(\F,\mc B)$.
\end{proposition}

\subsection{Camacho-Sad index \cite{CS}}
Let $C=\{f=0\}$ be an invariant curve of $\F$ and consider the decomposition given in (\ref{descomposicion}). We define the \textbf{Camacho-Sad index} of $\F$ along $C$ as
$$
CS_0(\F,C) = -\frac{1}{2\pi i}\int_{\partial C}\frac{1}{h}\eta.
$$
In contrast to the GSV-index, see \eqref{GSV en suma}, if $C_1$ and $C_2$ are $\mathcal{F}$-invariant curves without common components, the following holds (cf. \cite{Brunella}):
\begin{equation}\label{CS en suma}
CS_0(\F,C_1+C_2)=CS_0(\F,C_1)+CS_0(\F,C_2)+2i_0(C_1,C_2).
\end{equation}
We extend the definition for divisors with polar part.
\begin{definition}
Let $C=C_0-C_\infty$ be an invariant reduced divisor with $C_0$ and $C_\infty$ effective. We define
\[CS_0(\F,C)=CS_0(\F,C_0)+CS_0(\F,C_\infty)-2i_0(C_0,C_\infty).\]
\end{definition}
For this index we have the following.
\begin{theorem}\label{CS}
Let $C$ be a reduced invariant divisor, then 
\[CS_0(\F,C)=\sum_{p\in \bar C\cap E}CS_p(\bar \F,\bar C_0+\bar C_\infty)+\langle -A^{-1}S_C,S_C\rangle.\]
In other words, the Camacho--Sad index of \( C \) is equal to the sum of the Camacho--Sad indices of each of its branches after reduction, plus the integer \( \langle -A^{-1}S_C, S_C \rangle \).
\end{theorem}
\begin{proof}
Assume first that $C$ is an irreducible invariant curve. If $\pi_1$ is the blow-up of $0$  then (cf. \cite{Brunella})
\[CS_0(\F,C)=CS_{p_1}(\bar\F_1,\bar C_1)+\nu_0(C)^2.\]
If $\pi_n\circ\cdots\circ\pi_1$ is a part of the reduction $\pi=\pi_m\circ\cdots\circ\pi_n\circ\cdots\circ\pi_1$ of $\F$ desingularizing $C$ then
\begin{align*}
CS_0(\F,C)&=CS_{p_n}(\bar \F_n,\bar C_n)
+\sum_{i=1}^n\nu_{p_{i-1}}(\bar C_i)^2
=CS_{p_n}(\bar \F_n,\bar C_n)+\sum_{i=1}^m\nu_{p_{i-1}}(\bar C_i)^2\\
&=CS_{p_n}(\bar \F_n,\bar C_n)+ \langle (F^{-1})^{\mathsf{T}}S_C,(F^{-1})^{\mathsf{T}}S_C  \rangle =CS_{p_n}(\bar \F_n,\bar C_n)+\langle -A^{-1}S_C,S_C\rangle.
\end{align*}
Using that $CS_0(\F,C_1+C_2)=CS_0(\F,C_1)+CS_0(\F,C_2)+2i_0(C_1,C_2)$ we obtain inductively that
\[CS_0(\F,C)=\sum_{i=1}^nCS_{p_i}(\bar \F,\bar C)+\langle -A^{-1}S_C,S_C\rangle\] 
holds for any invariant reduced effective divisor. Consider now $C=C_0 - C_{\infty}$, then by definition we have
\begin{align*}
CS_0(\F,C)=&\sum_{p\in \bar C_0\cap E}CS_p(\bar\F,\bar C_0)+\langle -A^{-1}S_{C_0},S_{C_0}\rangle+
\sum_{p\in \bar C_\infty\cap E}CS_p(\bar\F,\bar C_\infty)+\langle -A^{-1}S_{C_\infty},S_{C_\infty}\rangle\\
&-2\langle -A^{-1}S_{C_0},S_{C_\infty}\rangle\\
&=\sum_{p\in \bar D\cap E}CS_p(\bar\F,\bar C_0+\bar C_\infty)+\langle -A^{-1}(S_{C_0}-S_{C_\infty}),S_{C_0}-S_{C_\infty}\rangle.
\end{align*}
It suffices to note that $\bar C=\bar C_0-\bar C_\infty$ so that $S_C=S_{C_0}-S_{C_\infty}$.
\end{proof}

\subsection{Variation index}
Let $C=\{f=0\}$ be a $\F$-invariant curve. On a pointed neighborhood 
%{\color{blue} $U^*= U\setminus{0}$}\footnote{Es necesario introducir $U$ y $U^*$?} 
of $0$ we may find a complex valued smooth $1$-form $\beta$, of type $(1,0)$, such that
\begin{equation}\label{descomposicion beta}
d\omega = \beta \wedge \omega.
\end{equation}
%We define
The \textbf{Variation index} of $\F$ along $C$ 
is defined as (\cite{K-S})
$$
Var_0(\F,C) = \frac{1}{2\pi i}\int_{\partial C} \beta.
$$
For any invariant curve we have the relation (cf. \cite[Proposition 5]{Brunella})
$$
Var_0(\F,C)= CS_0(\F,C)+ GSV_0(\F,C).
$$
This index is additive in the separatrices of $\F$
$$
Var_0(\F,C_1+C_2)= Var(\F,C_1)+ Var_0(\F,C_2). 
$$
So we extend $Var_0$ by linearity for arbitrary invariant reduced divisors. For this index we have the following.

\begin{theorem}\label{Variation}
Let $C=C_0 - C_{\infty}$ be an invariant reduced divisor, then
\[Var_0(\F,C)=\sum_{p\in E\cap\bar C_0}Var_p(\bar\F,\bar C_0)-\sum_{p\in E\cap\bar C_\infty}Var_p(\bar\F,\bar C_\infty)+\langle -A^{-1}S_{\mc B}-\iota,S_C\rangle.\]
\end{theorem}

\begin{proof}
For an invariant curve the formula is a direct consequence of Theorems \ref{CS} and \ref{GSV}. For a divisor with polar part we just apply the linearity of the index. 
\end{proof}

\subsection{Baum-Bott index}
Using the writing \eqref{descomposicion beta},
%As before we can write, on a pointed neighborhood %{\color{blue} $U^*= U\setminus{0}$} 
%of $0$, $d\omega = \beta \wedge \omega$. 
the \textbf{Baum-Bott} index of $\F$ at $0$ is 
$$
BB_0(\F)=\frac{1}{(2 \pi i)^2}\int_{S^3} \beta \wedge d\beta
$$
where $S^3$ in a small sphere around $0$, oriented as a boundary of a small ball containing $0$. For a non-degenerate reduced singularity with eigenvalues $\lambda_1$ and $\lambda_2$ we have
$$
BB_0(\F)= \frac{\lambda_1}{\lambda_2}+\frac{\lambda_2}{\lambda_1}+2.
$$
For the Baum-Bott index we have the following.
\begin{theorem}\label{BB}
Let $\pi$ be a reduction of singularities for $\F$ and $\bar{\F}$ the pull-back foliation, then
$$
BB_0(\F)=\sum_{p\in\mr{Sing}(\bar\F)}BB_p(\bar\F)+\langle -A^{-1}S_{\mc B},S_{\mc B}\rangle -2\langle S_{\mc B},\iota\rangle-\langle A\iota,\iota\rangle.
$$
\end{theorem}
\begin{proof}
Let $\pi_1$ be the blow-up of the origin and $\omega$ a $1$-form defining $\F$ with an isolated zero at $0$. Since the vanishing order of $\pi_1^*\omega$  along $E_1$ is $\ell_1$ and it is a global section of $N_{\bar\F_1}^*$ we have that $N^*_{\bar\F_1}=\pi^*N_\F^*\otimes\mc O(\ell_1 E_1)$, so that
$N_{\bar \F_1}=\pi^*N_\F\otimes\mc O(-\ell_1 E)$ and from \cite[Proposition 1]{Brunella} we get
\[BB_0(\F)=\sum_{p\in\mr{Sing}(\bar\F_1)}BB_p(\bar \F_1)+\ell_1^2.\]
By induction we deduce that
\begin{align*}
BB_0(\F)&=\sum_{p\in\mr{Sing}(\bar\F)}BB_p(\bar\F)+\langle\ell,\ell\rangle\\
&=\sum_{p\in\mr{Sing}(\bar\F)}BB_p(\bar\F)+\langle -A^{-1}S_{\mc B},S_{\mc B}\rangle -2\langle S_{\mc B},\iota\rangle-\langle A\iota,\iota\rangle
\end{align*}
thanks to Theorem~\ref{ell}. 
\end{proof}

Now we are able to give an extension of  \cite[Proposition 9]{Brunella}, see also  \cite[Theorem I]{FM}.
\begin{corollary}
Let $\B$ be a balanced divisor of $\F$, then $BB_0(\F)=CS_0(\F,\mc B)=Var_0(\F,\mc B)$.
\end{corollary}
\begin{proof}
Let us denote by $C$ the union of the isolated separatrices of $\F$. Since the strict transforms of the dicritical separatrices meet $E$ at regular points of $\bar\F$ we have
\[CS_0(\F,\mc B)=\sum_{p\in\bar C\cap E}CS_p(\bar\F,\bar C)+\langle -A^{-1}S_{\mc B},S_{\mc B}\rangle\]
and
\[Var_0(\F,\mc B)=\sum_{p\in\bar C\cap E}Var_p(\bar\F,\bar C)+\langle -A^{-1}S_{\mc B},S_{\mc B}\rangle-\underbrace{\langle\iota,S_{\mc B}\rangle}_{\langle\iota,S_C\rangle}.\]
Then
\begin{align*}
CS_0(\F,\mc B)-Var_0(\F,\mc B)&=\sum_{p\in E\cap \bar C}\big(CS_p(\bar\F,\bar C)-Var_p(\bar\F,\bar C)\big)+\langle S_C,\iota\rangle\\
&=-\sum_{p\in E\cap\bar C}GSV_p(\bar\F,\bar C)+\langle S_C,\iota\rangle=0
\end{align*}
because $GSV_p(\bar\F,\bar C)=1$ for each $p\in E\cap\bar C$, as $\F$ is a generalized curve. It remains to prove the first equality. Set $\bar C\cap E=\{p_1,\ldots,p_M\}$ and $\lambda_j=CS_0(\bar\F,\bar C,p_j)$ for each $j=1,\ldots,M$. Let us denote by \( E^i_1, \ldots, E^i_k \) the invariant components of \( E \), and let \( q_1, \ldots, q_N \) be the intersection points between any two distinct components among them. We have that $BB_{q_i}(\bar\F)=\mu_i+\frac{1}{\mu_i}+2$ for $i=1,\ldots,N$, where $\mu_i,\frac{1}{\mu_i}$ are the Camacho-Sad indices of $\bar\F$ at $q_i$ along the two invariant irreducible components of $E$ meeting at $q_i$. We have
\begin{align*}
BB_0(\F)-CS_0(\F,\B)&=\sum_{j=1}^M\big(BB_{p_j}(\bar\F)-CS_{p_j}(\bar\F,\bar C)\big)+\sum_{i=1}^NBB_{q_i}(\bar\F)-2\langle S_{\mc B},\iota\rangle-\langle\iota,A\iota\rangle\\
&=\sum_{j=1}^M\Big(\frac{1}{\lambda_j}+2\Big)-2\langle S_{\mc B},\iota\rangle+\sum_{i=1}^N\Big(\mu_i+\frac{1}{\mu_i}+2\Big)-\langle\iota,A\iota\rangle\\
&=\sum_{j=1}^M\frac{1}{\lambda_j}+\sum_{i=1}^N\Big(\mu_i+\frac{1}{\mu_i}+2\Big)-(E^i_1+\cdots+E^i_k)^2.
\end{align*}
By the Camacho-Sad index theorem we have that 
$$
\sum_{j=1}^M\frac{1}{\lambda_j}+\sum_{i=1}^N\Big(\mu_i+\frac{1}{\mu_i}\Big)=(E^i_1)^2+\cdots+(E^i_k)^2.
$$
Moreover $\sum_{r,s}E^i_r \cdot E^i_s = 2N$. We conclude that $BB_0(\F)-CS_0(\F,\B)=0$.
\end{proof}

\section{The bifurcation conjecture}

Let \( f, g \in \mathcal{O}_{\mathbb{C}^2, 0} \) be germs vanishing at the origin and coprime. Recall from the Introduction that the Milnor number of the pair, $\mu(f,g)$, coincides with the Milnor number of the pencil foliation $\F_{f/g}$ defined by the form $\omega= fdg-gdf$. Furthermore, we state the following well-known result, whose proof is left to the reader.
\begin{lemma}\label{lemma_red}
Let \( f, g \in \mathcal{O}_{\mathbb{C}^2, 0} \) be germs vanishing at the origin and coprime. Then, the following statements are equivalent.
\begin{enumerate}
\item $0$ is an isolated singularity of $\omega$.
\item $\mu(f, g) = \mu_0(\mathcal{F}_{f/g})$.
\item every element of the pencil generated by \( f \) and \( g \) is reduced.
\end{enumerate}
\end{lemma}

Since the case $\mu(f,g)=\infty$ has already been addressed in \cite{Sa}, we  assume from now that $\mu(f,g)<\infty$. In particular, the foliation $\F = \F_{f/g}$  defined by the 1-form $\omega = fdg - gdf$ has an isolated singularity at the origin and  all elements of the pencil are reduced.

Before giving the proofs of Bifurcation formula we establish the following key lemmas. The first says that bifurcation fibers are those containing some isolated separatrix of $\F_{f/g}$  and the second one gives a way to construct balanced divisors for pencil type foliations.

\begin{lemma}\label{isolated separatrices}
A fiber of the pencil is a bifurcation fiber if and only if it contains an isolated separatrix.
\end{lemma}
\begin{proof}
Let \( \pi: X \to (\mathbb{C}^2, 0) \) be a reduction of singularities for the foliation \( \mathcal{F}_{f/g} \), let \( F \) be a fiber without isolated components, and let \( \overline{F} \) denote its strict transform. Then every branch of \( \overline{F} \)  intersects \( E \) transversely and only in dicritical components. Therefore, if \( G \) is a generic fiber close enough to \( F \), there exists a one-to-one correspondence between branches of \( \overline{G} \) and branches of \( \overline{F} \) in such a way that corresponding branches meet transversely the same dicritical component of the exceptional divisor. Thus \( S_F = S_G \) and hence they have the same Milnor number.

Reciprocally, if $F$ is a fiber that contains some isolated separatrix then $F$ is not equisingular with the generic fiber. It follows from \cite[Corollary 7.3.9]{Casas} that $\mu_0(F) >\mu_{gen}$.
\end{proof}

\begin{lemma}\label{balanceado por fibras}
Let $h_1, \ldots , h_r$ be the bifurcation fibers of the pencil $\F_{f/g}$ and assume that $f$, $g$ and $\varphi_1, \ldots, \varphi_r$ are generic fibers. Then the divisor $\B=(h_1)+\cdots+(h_r)+(f)+(g) -(\varphi_1)-\cdots - (\varphi_r)$ is balanced.
\end{lemma}

\begin{proof}
Let \( \pi: X \to (\mathbb{C}^2, 0) \) be a reduction of singularities for the foliation \( \mathcal{F}_{f/g} \). Then the rational map \( F := (f/g) \circ \pi : X \to \mathbb{P}^1 \) has no indeterminacy points. We can write $(h_j) = (f/g)^{-1}(c_j), \quad j = 1, \ldots, r,$ where each \( c_j \in \mathbb{C}^* \), since \( f \) and \( g \) define generic fibers of the pencil. Observe that $F$ contracts all invariant divisors and, for any dicritical divisor $D$, the map $F_D =F|_{D}:D \to \mathbb{P}^1$ is dominant. Then we can write $F_D^{-1}(c_j) = \{p_1^j, \ldots, p_{k_j}^j, q_1^j, \ldots, q_{\ell_j}^j\}$, where $p_i^j$ are regular points and $q_i^j$ are corners in $D$. In particular $m(F_D, p_i^j)=1$ and $m(F_D, q_i^j) \geq 1$. Thus, for each $j=1, \ldots, r$
\begin{equation}\label{multiplicidad en D}
\deg(F_D)= k_j + \sum_{i=1}^{\ell_j} m(F_D, q_i^j).
\end{equation}
By the Riemann-Hurwitz theorem applied to $F_D$ we also have
$$
2=2\deg(F_D) - \sum_{i,j}(m(F_D, q_i^j)-1)=2\deg(F_D) - \sum_{i,j}m(F_D, q_i^j) + \ell_1 + \cdots + \ell_r.
$$
Since the points $q_i^j$ are different corners in $D$ we have that $\mr{val}_D \leq \ell_1 + \cdots + \ell_r $. On the other hand, let $q$ be any corner in $D$ and let us show that $q$ belongs to $\{q_i^j\}$ --- this will mean that $\mr{val}_D = \ell_1 + \cdots + \ell_r $. The point $q$ belongs to an invariant irreducible component $I$ of the exceptional divisor. It is known (see \cite{O-B,Rosas}) that there is an isolated separatrix whose strict transform meets $I$. Thus, by Lemma~\ref{balanceado por fibras}, this separatrix is contained in a bifurcation fiber $(h_j)$. So, we conclude that $q \in I \subset F^{-1}(c_j)$ and therefore $q \in \{q_i^j\}$ . So that
\begin{equation}\label{Riemann-Hurwitz}
2\deg(F_D) - \sum_{i,j}m(F_D, q_i^j) =2 - \mr{val}_D.
\end{equation}
Finally, denoting by $Sep(D)$ the set of separatrices of $\F$ whose strict transforms meet $D$, we have
\begin{align*}
\sum_{B \in Sep(D)} a_B &= k_1 + \cdots+ k_r + (2-r) \deg(F_D)\\
& \underset{\text{eq.} (\ref{multiplicidad en D})}{=} 2 \deg(F_D) - \sum_{i,j}m(F_D, q_i^j)\\
& \underset{\text{eq.} (\ref{Riemann-Hurwitz})}{=} 2-\mr{val}_D,
\end{align*}
which proves that $\B$ is balanced.
\end{proof}

\begin{example}(Genzmer)\label{Ejemplo-Genzmer}
Consider the pencil, absolutely dicritical defined by $f = xy+ y^2 + x^3=f_1f_2$ and $g=xy=g_1g_2$. Then the only bifurcation fiber is $h=f-g = y^2+x^3$ with $\mu(h) =2 $ and $\mu_{gen}=1$. Its desingularization is $\pi=\pi_3\circ\pi_2\circ\pi_1$, where $\pi_1$ is the blow-up of the origin $0$,
$\pi_2$ is the blow-up of the intersection point $p_1$ of $E_1^1=\pi_1^{-1}(0)$ with the strict transform of $y=0$ and $\pi_3$ is the blow-up of the point $p_2=E_1^2\cap E_2^2$, $E_1^2$ is the strict transform of $E_1^1$ and $E_2^2=\pi_2^{-1}(p_1)$. The exceptional divisor is $E=\pi^{-1}(0)=E_1\cup E_2\cup E_3$, where $E_1=E_1^3$ and $E_2=E_2^3$ are the strict transforms of $E_1^2$ and $E_2^2$ by $\pi_3$ and $E_3=E_3^3=\pi_3^{-1}(p_2)$. Thus, the matrix $F$ and the intersection matrix $A$ associated to $\pi$ are 
\[F=\left(\begin{array}{rrr}
1 & 0 & 0\\
-1 & 1 & 0\\
-1 & -1 & 1\end{array}\right)\quad\text{and}\quad A=-F^{\mathsf{T}}F=\left(\begin{array}{rrr}-3 & 0 & 1\\ 0 & -2 & 1\\ 1 & 1 & -1\end{array}\right).\]
The dual graph of the desingularization is the following:
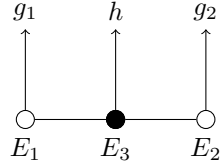
\begin{center}
\begin{tikzpicture}[scale=1.2]
\draw (-1,0) to (1,0);
\foreach \i in {-1,0,1}{\draw[->] (\i,0) to (\i,1);\draw[fill,white] (\i,0) circle [radius=0.1];\draw (\i,0) circle [radius=0.1];}
\draw [fill,black] (0,0) circle [radius=0.1];
\node [above] at (-1,1) {$g_1$};
\node [above] at (0,1) {$h$};
\node [above] at (1,1) {$g_2$};
\node [below] at (-1,-.1) {$E_1$};
\node [below] at (0,-.1) {$E_3$};
\node [below] at (1,-.1) {$E_2$};
\end{tikzpicture}
\end{center}
Each dicritical divisor $D\in \{E_1, E_3\}$ has valence $\mr{val}_D = 1$. Thus the divisor $\B= (h) + (f) + (g) - (\varphi)$ is balanced, where $\varphi = \varphi_1 \varphi_2$ is a generic fiber and $f_1, \varphi_1$ (respectively $f_2, \varphi_2$) comes from separatrices passing thought $E_1$ (respectively $E_3$). In fact
$$
\sum_{B\in Sep(D)}a_B = 1+ 1 - 1 = \mr{val}_D.
$$
\end{example}

\begin{example}\label{ejemplo mald.} (Szawlowski): Take the pencil generated by
$f=(x^3+y^5)+y(y^2-3x^2)=f_1 f_2 f_3$ and $g=y(y^2-3x^2)=g_1 g_2 g_3$. Then bifurcation fibers are $h_1=f-g$, $h_2=2f - g$ and $h_3=2f-3g$ with $\mu(h_1)=8$, $\mu(h_2)= \mu(h_3) = 6$ and $\mu_{gen}=4$. The dual graph of the desingularization is the following: 
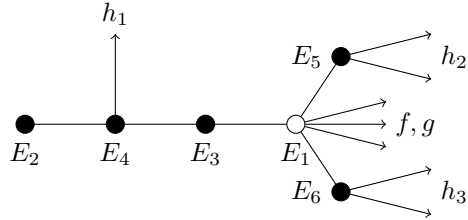
\begin{center}
\begin{tikzpicture}[scale=1.2]
\draw (-2,0) to (1,0);
\foreach \i in {-2,...,0}{\draw [fill,black] (\i,0) circle [radius=0.1];}
\draw (1.5,.75) to (1,0) to (1.5,-.75);
\draw [fill,black] (1.5,.75) circle [radius=0.1];
\draw [fill,black] (1.5,-.75) circle [radius=0.1];
\draw[->] (-1,0) to (-1,1);
\draw[->] (1.5,0.75) to (2.5,1); 
\draw[->] (1.5,0.75) to (2.5,.5); 
\draw[->] (1.5,-0.75) to (2.5,-1); 
\draw[->] (1.5,-0.75) to (2.5,-.5); 
\node[above] at (-1,1) {$h_1$};
\node[right] at (2.5,.75) {$h_2$};
\node[right] at (2.5,-.75) {$h_3$};
\draw[->] (1,0) to (2,.25);
\draw[->] (1,0) to (2,-.25);
\draw[->] (1,0) to (2,0);
\node [right] at (2,0) {$f,g$};
\node [below] at (-2,-0.1) {$E_2$};
\node [below] at (-1,-0.1) {$E_4$};
\node [below] at (0,-0.1) {$E_3$};
\node [below] at (1,-0.1) {$E_1$};
\node [left] at (1.4,0.75) {$E_5$};
\node [left] at (1.4,-0.75) {$E_6$};
\draw [fill,white] (1,0) circle [radius=0.1];
\draw  (1,0) circle [radius=0.1];
\end{tikzpicture}
\end{center}
We have only one dicritical component $E_1$ with $\mr{val}_{E_1} = 3$ and $\B = (h_1)+(h_2)+(h_3) + (f) + (g) - (\varphi_1)- (\varphi_2) - (\varphi_3)$. 
Observe that in this case the strict transform of the generic fiber meets $E_1$ in three points, the strict transform of the bifurcation fiber $h_1$ only meets $E_4$ once and the strict transform of the bifurcation fiber $h_2$ (resp. $h_3$) meets twice  $E_5$ (resp. $E_6$) and once the dicritical component $E_1$.
Thus
$$
\sum_{B\in Sep(E_1)}a_B= \underset{h_2}{\underbrace{1}}+ \underset{h_3}{\underbrace{1}}+ \underset{f}{\underbrace{3}}+ \underset{g}{\underbrace{3}} - \underset{\varphi_1}{\underbrace{3}} - \underset{\varphi_2}{\underbrace{3}} - \underset{\varphi_3}{\underbrace{3}} = 2 - \mr{val}_{E_1}.
$$
\end{example}

\subsection{First proof of the Bifurcation Formula}

Now we can state the main result of this section.

\begin{theorem}[Bifurcation formula]\label{Bifurcation}
Assume that $f,g\in\C\{x,y\}$ both vanish at the origin and are coprime. Then
\begin{equation}\label{Bif}
\displaystyle\mu(f,g)=\mu_0(fg)+\sum_{[\alpha:\beta]\in\mathcal{B}(f,g)^{*}}\left(\mu_0(\alpha f +\beta g)-\mu_{gen}(f,g)\right)
\end{equation}
\end{theorem}
\begin{proof}
Since the equality does not depend on the generators, we can assume that $f$ and $g$ are generic fibers of the pencil. By Lemma~\ref{balanceado por fibras}
\[\mc B=\sum_{i=1}^r(h_i)+(f)+(g)-\sum_{i=1}^r(\varphi_i)\] 
is a balanced divisor of $\F$. Consider a reduction of singularities of the pencil, let $A$ and $F$ be the corresponding intersection and Cholesky-type matrices and define the function 
\[M(S):=\langle -A^{-1}S,S\rangle-\langle S,(I+F^{-1})u\rangle+1.\]
We remark that if $h$ is any invariant curve then $M(S_h) = \mu_0(h)$. Moreover, notice  that  $M(S_{\B})=\mu_0(\F_{f/g})=\mu(f,g)$, by Theorem \ref{muF}, and $\mu(fg)=M(S_{fg})=M(2S_g)$. Then the bifurcation formula (\ref{Bif}) is equivalent to
\begin{equation}\label{FM}
M\Big(\sum_{i=1}^rS_{h_i}+(2-r)S_g\Big)=M(2S_g)+\sum_{i=1}^r(M(S_{h_i})-M(S_g)),
\end{equation}
A direct computation leads to the equivalent equality
\begin{equation*}\label{i0}
2\sum_{1\le i<j\le r}\underbrace{\langle -A^{-1}S_{h_i},S_{h_j}\rangle}_{h_i\cdot h_j}
+2(2-r)\sum_{i=1}^r\underbrace{\langle -A^{-1}S_{h_i},S_g\rangle}_{h_i\cdot g}+r(r-3)\underbrace{\langle -A^{-1}S_g,S_g\rangle}_{f \cdot g}
=0,
\end{equation*}
which holds because $i_0(h_i,h_j)=i_0(h_i,g)=i_0(f,g).$
\end{proof}

\subsection{Second proof of the Bifurcation Formula}

We now present a second proof of Theorem \ref{Bifurcation} which is an application of Proposition \ref{mil = intmil}. We start with a proposition.
\begin{proposition}\label{prop_1}
Assume that $f,g\in \mathcal O_{\C^2,0}$ both vanish at the origin and are coprime. Let $C_{\alpha,\beta}:\alpha\cdot f+\beta\cdot g=0$, then
\[GSV_0(\F_{f/g},C_{\alpha,\beta})=i_0(f,g). \]
\end{proposition}
\begin{proof}
First of all note that (from definition of GSV-index) it is clear that $GSV_0(\F,\{f=0\})=i_0(f,g)$. On the other hand, we can choose $\alpha\cdot f+\beta\cdot g $ and $g$ as generators of the pencil, thus $GSV_0(\F_{f/g},C_{\alpha,\beta})=i_0(\alpha\cdot f+\beta\cdot g,g)=i_0(f,g)$ and the result follows.
\end{proof}
\begin{remark}\label{IntMil-GSV}
It follows from Remark \ref{GSV vs Milnor} that  Proposition~\ref{prop_1} is equivalent to the equality $\mu_0\left(\F_{f/g},C_{\alpha,\beta}\right)=i_0(f,g)+\mu_0(C_{\alpha,\beta})$.
\end{remark}

\noindent \textit{Second proof of the Bifurcation formula.} Once again, from Lemma \ref{balanceado por fibras} we can consider the balanced divisor $\B$ given by
$$
\B=(h_1)+\cdots+(h_r)+(f)+(g) -(\varphi_1)-\cdots - (\varphi_r),
$$
where $h_1, \ldots, h_r$ are the bifurcation fibers and $f$, $g$, $\varphi_1, \ldots, \varphi_r$ are generic. Thus, according to Proposition~\ref{mil = intmil} and  \eqref{eq:Milnor along a divisor},
$$
\mu_0(\F) = \mu_0(\F, \B)= \sum_{i=1}^r\mu_0(\F,(h_i)) + \mu_0(\F,(f))+\mu_0(\F,(g)) - \sum_{i=1}^r\mu_0(\F,(\varphi_i)) -1.
$$
By Remark \ref{IntMil-GSV} we have that $\mu_0(\F,(h_i))=i_0(f,g)+\mu(h_i)$, $\mu_0(\F,(f))=i_0(f,g)+\mu(f)$, $\mu_0(\F,(g))=i_0(f,g)+\mu(g)$ and $\mu_0(\F,(\varphi_i))=i_0(f,g)+\mu_{gen}$ , therefore
$$
\mu_0(\F)= 2 i_0(f,g) +\mu_0(f)+\mu_0(g) -1 + \sum_{i=1}^r (\mu_0(h_i)-\mu_{gen})= \mu_0 (fg) + \sum_{i=1}^r (\mu_0(h_i)-\mu_{gen})
$$
and we conclude the proof. \qed

\subsection{Semitame functions}
We recall that a meromorphic function $(f:g): (\C^2,0)\to \mathbb{P}^1$ is called \textbf{semitame} in \cite{Bodin-Pichon} and \cite{Sa} if the only possible bifurcation fibers are $f$ and $g$. A direct consequence of the Bifurcation Formula is the following.

\begin{corollary}
If $f$ and $g$ are coprime and all fibers of the pencil are reduced, then $\mu(f,g)\geq \mu_0(fg)$ with equality if and only if the germ $(f:g)$ is semitame.
\end{corollary}

Observe that, in general, a pencil $\mathcal{F}_{f/g}$ may have three non-reduced fibers, including those defined by $f$ and $g$. In this case, it would not be semitame. Nevertheless $\mu(f, g)=\mu_0(fg)=\infty$.

\subsection{Universal unfoldings of meromorphic functions}
Suwa \cite{Suwa} studied unfoldings of germs of meromorphic functions. He showed that for a germ of a meromorphic function $f/g$, universal unfoldings exist under a finiteness condition. This condition is related to the finiteness of the number
\[\nu(f,g):=\dim_{\C}\frac{\langle f,g\rangle}{\langle f_x g-g_x f, f_y g-g_y f\rangle}<\infty.\]
This number is also the dimension of the parameter space of the semi-universal unfolding in the meromorphic context.
\par As a consequence of the Bifurcation Formula (Theorem \ref{Bifurcation}) and since $\nu(f,g)=\mu(f,g)-i_0(f,g)$, we obtain the following corollary which answers a question posed by Szawlowski \cite[Section 5]{Sa}.

\begin{corollary}
If $f$ and $g$ are coprime and all fibers of the pencil are reduced, then
\[\nu(f,g)=\mu_0(fg)-i_0(f,g)+\sum_{[\alpha:\beta]\in\mathcal{B}(f,g)^{*}}\left(\mu_0(\alpha f +\beta g)-\mu_{gen}(f,g)\right).\]
\end{corollary}

\end{document}